\documentclass[10 pt]{amsart}

\usepackage[colorlinks=true,linkcolor=blue,citecolor=blue,urlcolor=blue,hypertexnames=false]{hyperref}
\usepackage{bookmark}
\usepackage{amsthm,thmtools,amssymb,amsmath,amscd}

\usepackage{fancyhdr}
\usepackage{esint}

\usepackage{enumerate}

\usepackage{pictexwd,dcpic}

\swapnumbers
\declaretheorem[name=Theorem,numberwithin=section]{thm}
\declaretheorem[name=Remark,style=remark,sibling=thm]{rem}
\declaretheorem[name=Lemma,sibling=thm]{lemma}

\declaretheorem[name=Corollary,sibling=thm]{cor}
\declaretheorem[name=Assumption,style=remark,sibling=thm]{ass}

\numberwithin{equation}{section}

\usepackage{cleveref}
\crefname{lemma}{Lemma}{Lemmata}
\crefname{prop}{Proposition}{Propositions}
\crefname{thm}{Theorem}{Theorems}
\crefname{cor}{Corollary}{Corollaries}
\crefname{defn}{Definition}{Definitions}
\crefname{example}{Example}{Examples}
\crefname{rem}{Remark}{Remarks}
\crefname{ass}{Assumption}{Assumptions}
\crefname{notation}{Notation}{Notation}

\usepackage{autonum}

\renewcommand{\~}{\tilde}
\renewcommand{\-}{\bar}
\newcommand{\bs}{\backslash}
\newcommand{\cn}{\colon}
\newcommand{\sub}{\subset}

\newcommand{\R}{\mathbb{R}}

\renewcommand{\S}{\mathbb{S}}
\renewcommand{\H}{\mathbb{H}}

\newcommand{\8}{\infty}

\renewcommand{\a}{\alpha}
\renewcommand{\b}{\beta}
\newcommand{\g}{\gamma}
\renewcommand{\d}{\delta}

\renewcommand{\k}{\kappa}
\renewcommand{\l}{\lambda}

\newcommand{\s}{\sigma}
\newcommand{\p}{\varphi}

\newcommand{\vt}{\vartheta}

\newcommand{\G}{\Gamma}

\newcommand{\del}{\partial}

\newcommand{\n}{\nabla}

\newcommand{\fr}[2]{\frac{#1}{#2}}
\newcommand{\x}{\times}

\DeclareMathOperator{\id}{id}

\DeclareMathOperator{\graph}{graph}

\DeclareMathOperator{\dist}{dist}

\newcommand{\Theo}[3]{\begin{#1}\label{#2} #3 \end{#1}}
\newcommand{\pf}[1]{\begin{proof} #1 \end{proof}}
\newcommand{\eq}[1]{\begin{equation}\begin{alignedat}{2} #1 \end{alignedat}\end{equation}}

\newcommand{\br}[1]{\left(#1\right)}

\newcommand{\ra}{\rightarrow}
\newcommand{\hra}{\hookrightarrow}

\newcommand{\mc}{\mathcal}

\newcommand{\mrm}{\mathrm}

\newcommand{\hp}{\hphantom}

\begin{document}

\title[Pinched flows in Euclidean and hyperbolic space]{Expansion of pinched hypersurfaces of the Euclidean and hyperbolic space by high powers of curvature}
\begin{abstract}
We prove convergence results for expanding curvature flows in the Euclidean and hyperbolic space. The flow speeds have the form $F^{-p}$, where $p>1$ and $F$ is a positive, strictly monotone and $1$-homogeneous curvature function. In particular this class includes the mean curvature $F=H$. We prove that a certain initial pinching condition is preserved and the properly rescaled hypersurfaces converge smoothly to the unit sphere. We show that an example due to Andrews-McCoy-Zheng can be used to construct strictly convex initial hypersurfaces, for which the inverse mean curvature flow to the power $p>1$ loses convexity, justifying the necessity to impose a certain pinching condition on the initial hypersurface.
\end{abstract}
\date{\today}

\author[H. Kr\"oner]{Heiko Kr\"oner }
\address{Albert-Ludwigs-Universit\"{a}t,
Mathematisches Institut, Ersnt-Zermelo-Str.~1, 79104
Freiburg, Germany}
\email{heiko.kroener@uni-hamburg.de}

\author[J. Scheuer]{Julian Scheuer }
\address{Albert-Ludwigs-Universit\"{a}t,
Mathematisches Institut, Ernst-Zermelo-Str.~1, 79104
Freiburg, Germany}
\email{julian.scheuer@math.uni-freiburg.de}
\thanks{This work has been supported by the "Deutsche Forschungsgemeinschaft" (DFG, German research foundation) within the research grant "Harnack inequalities for curvature flows and applications", grant number SCHE 1879/1-1.
}
\maketitle

%\tableofcontents

\section{Introduction}

We consider inverse curvature flows in either the Euclidean space $N=\R^{n+1} $ or the hyperbolic space $N=\H^{n+1}$ with sectional curvature $K_N=-1$, i.e. a family of embeddings
\eq{x\cn [0,T)\x M\ra N,}
where $M$ is a closed, connected and orientable manifold of dimension $n$, which solves
\eq{\label{flow}\dot{x}=\fr{1}{F^p}\nu,\quad 1<p<\infty.} Here $\nu=\nu(t,\xi)$ is the outward pointing normal to the flow hypersurface $M_t=x(t,M)$ and $F$ is evaluated at the principal curvatures $\k_i$ at the point $x(t,\xi)$.

Let us first state our main theorem. Therefore we will use the following assumptions for the curvature function $F$ and the initial hypersurface $M_0$.

\begin{ass} \label{1}
Let $n\ge 2$ and 
\begin{equation}
\Gamma_{+} = \{(\kappa_i)\in \mathbb{R}^n: \kappa_i>0, 1 \le i \le n\}.
\end{equation}
Let $F \in C^{\infty}(\Gamma_{+})$ be a positive, strictly monotone and symmetric curvature function which is homogeneous of degree one and normalized to $F(1, ..., 1)=n$. 
\end{ass}
\begin{ass} \label{5}
Let $F$ satisfy Assumption \ref{1}. Let
\begin{equation}
x_0: M \rightarrow M_0 \subset N
\end{equation}
be the smooth embedding of a hypersurface $M_{0}$ with $F_{|M_0}>0,$ which can be written as a graph over a geodesic sphere $\mathbb{S}^n$,
\begin{equation}
M_0 = \{(u(0, y),y); y \in \mathbb{S}^n\}.
\end{equation}
\end{ass}

In the following statement of our main theorem, we let 
\eq{H=\sum_{i=1}^n \k_i}
be the mean curvature, $A$ the Weingarten operator, $\|A\|$ be its norm with respect to the induced metric and $I$ the identity. The principal curvatures $\k_{i}$, $1\leq i\leq n$, will always be labelled according to 
\eq{\k_{1}\leq \dots\leq \k_{n}.}

\Theo{thm}{main_result}{
Let $n\ge 2$ and let $N=N^{n+1}$ be either the Euclidean space or the hyperbolic space of constant sectional curvature $K_N=-1.$ Let $p>1$ and let $F$, $M_0$ satisfy \Cref{1,,5}.
Furthermore, we assume that in case $K_N=0$,
$M_0$ satisfies the pinching condition
\begin{equation} \label{3}
\|A\|^2-\frac{1}{n}H^2 < c_0H^2,
\end{equation}
and in case $K_N=-1$, that $M_0$ satisfies the pinching condition
\eq{\label{main-2}\|A-I\|^2-\fr 1n (H-n)^2< c_0 (H-n)^2,}
 where $0<c_0=c_0(F,n,p)<\frac{1}{n(n-1)}$ is sufficiently small.
Then:
\begin{enumerate}
\item[(i)] There exists a unique smooth solution on a maximal time interval
\begin{equation} \label{2}
x: [0, T^{*})\times M \rightarrow N,
\end{equation}
of the equation
\begin{equation} \label{4}
\begin{aligned}
\dot x =& \frac{1}{F^p}\nu  \\
x(0, \xi) =& x_0(\xi),
\end{aligned}
\end{equation}
where $T^*<\8$ in case $K_N=0$ and $T^*=\8$ in case $K_N=-1$,
where $\nu = \nu(t, \xi)$ is the outward unit normal to $M_t=x(t, M)$ at $x(t, \xi)$ and $F$ is evaluated at the principal curvatures of $M_t$ at $x(t, \xi)$.

\item[(ii)] The flow hypersurfaces $M_t$ can be written as graphs of a function $u=u(t, \cdot)$ over $\mathbb{S}^n$ so that
\begin{equation}
\lim_{t\rightarrow T^{*}}\inf_{\mathbb{S}^n}u(t, \cdot) =\infty
\end{equation}
and properly rescaled flow hypersurfaces converge for all $m \in \mathbb{N}$ in $C^m(\mathbb{S}^n)$ to a geodesic sphere.

\item[(iii)] In case $K_N=0$ there exists a point $Q \in \mathbb{R}^{n+1}$ and a sphere $S^{*}=S_{R^{*}}(Q)$ around $Q$ with radius
$R^{*}$ such that the spherical solutions $S_t$ with radii $R_t$ of \eqref{4} with $M_0=S_{R^{*}}$ satisfy
\begin{equation}
\dist (M_t, S_t) \le c R_t^{-\frac{p}{2}} \quad \forall \ t \in [0, T^{*}),
\end{equation}
$c=c(p, M_0, F)$. Here $\dist$ denotes the Hausdorff distance of compact sets.
\end{enumerate}
}

\begin{rem}
(i) The rescalings, mentioned in (ii) of this theorem, are given by
\eq{\~u=\fr{u}{t}}
in case $K_N=-1$ and by
\eq{\~u=\fr{u}{\Theta}}
in case $K_N=0,$ where $\Theta$ is the unique radius of a sphere which exists exactly as long as the flow, i.e. for the time $T^*.$

(ii) A result similar to \Cref{main_result} (iii) can not be deduced in the hyperbolic space, cf. the nice counterexample in \cite{HungWang:09/2015}.

(iii) For flows by high powers of curvature, pinching conditions similar to \eqref{main-2} and \eqref{4} have already appeared for contracting flows in \cite{AndrewsMcCoy:03/2012} and \cite{Schulze:/2006}. Indeed we can mimic a counterexample to preserved convexity for contracting flows by Andrews-McCoy-Zheng \cite[Thm.~3]{AndrewsMcCoyZheng:07/2013} and show that in general strict convexity (in particular general pinching) will be lost if $p>1$ and $F=H$.
\end{rem}

Theorem \ref{main_result} shows the following: If we remove the assumptions $\Gamma=\Gamma_{+}$, $F_{|\partial \Gamma}=0$ as well as the concavity of $F$  in \cite[Theorem 1.2]{Gerhardt:01/2014} and \cite[Theorem 1.2]{Scheuer:05/2015} and replace them by a certain pinching condition for the initial hypersurface then the resulting theorems are true. This allows us e.g. to consider the interesting case $F=H$ and more generally
\eq{F=H_k^{\fr 1k},}
where $H_k$ is the $k$-th elementary symmetric polynomial.

Expanding curvature flows of the form \eqref{flow} with $p=1$ have attracted a lot of attention, since they have proven to be useful in the deduction and generalization of several geometric inequalities, like the Riemannian Penrose inequality, \cite{HuiskenIlmanen:/2001}, some Minkowski type  inequalities,  \cite{BrendleHungWang:01/2016} and \cite{GuanLi:08/2009}, and Alexandrov-Fenchel type inequalities as in \cite{GeWangWu:04/2014}. The case $p=1$ has been treated in broad generality, cf. \cite{Gerhardt:/1990} and \cite{Urbas:/1990,Urbas:/1991} for the Euclidean case, \cite{Gerhardt:11/2011} for the hyperbolic case and \cite{ChenMao:04/2018,Ding:01/2011,Gerhardt:/2015,Kroner:/2018,LiWei:/2017,Lu:09/2016,MakowskiScheuer:11/2016,Scheuer:01/2017,Zhou:04/2018} for other ambient spaces. For the case $p\neq 1$ there are fewer results. In the Euclidean case there is \cite{Gerhardt:01/2014,Li:06/2010,Schnuerer:/2006} and in the hyperbolic case there is \cite{Scheuer:05/2015}. All of these papers share the fact that only curvature functions $F$ are considered which vanish on the boundary of $\G_{+}$, a property which forces preservation of convexity by definition. The goal of the present paper is the generalisation of both of these works. The second author already obtained an improvement of \cite{Scheuer:05/2015} in \cite{Scheuer:01/2015}, where he could drop the pinching assumption on the initial hypersurface for some powers $p>1$ of the inverse Gauss curvature flow. But in particular for powers $p>1$ of the inverse mean curvature flow there are no results, up to the authors' knowledge. In the recent preprint \cite{LiWangWei:09/2016} Li, Wang and Wei proved convergence results for the case $p<1$ in $\R^{3}$ and $\H^{3}$. The novelty in this paper is that they could also treat non-concave curvature functions, since for parabolic equations in two variables one can replace the Krylov-Safonov estimates by a regularity result due to Andrews \cite{Andrews:/2004}. Wei provided some new pinching estimates in the cases $p<1$ for a broad class of curvature functions in \cite{Wei:/2018}.   Some non-homogeneous flow speeds are considered in \cite{ChowTsai:/1997} and \cite{ChowTsai:/1998}.

The paper is organised as follows. In \cref{Not,Ev} we collect some notation and evolution equations. In \cref{CEx} we give the counterexample. In \cref{PE} we prove the crucial preservation of a specific pinching of the initial hypersurface, whereafter in \cref{EC,HC} we give an outline  on how to finish the proof of \cref{main_result}. Here one could basically follow the lines of \cite{Gerhardt:01/2014} and \cite{Scheuer:05/2015}, but due to the pinching estimates several aspects of the proofs in these references simplify so we present these simplified arguments for convenience.

\section{Notation and preliminaries}\label{Not}

Let $n\geq 2$ and $N=N^{n+1}$ be either the Euclidean space or the hyperbolic space with constant sectional curvature $K_N=-1$ of dimension $n+1$. Let $M=M^n$ be a compact, connected, smooth manifold and 
\eq{x\cn M\hra N}
be an embedding with unit outward normal vector field $\nu,$
(compare the nice note \cite{Samelson:07/1969}).
Let $g=(g_{ij})$ be the induced metric on $M$, where $g_{ij}$ are the components of $g$ with respect to the basis $x_i=\del_i x$, $1\leq i\leq n$. In tensor expressions latin indices always range between $1$ and $n$ and greek indices range from $0$ to $n$ indicating components of tensors of the ambient space. The coordinate expression of a covariant derivative with respect to the Levi-Civita connection of $g$ of a tensor field $T\in T^{k,l}M$ are indicated by a semi-colon,
\eq{\n^r T=\br{T^{i_1\dots i_k}_{j_1\dots j_l;m_1\dots m_r}}.}
The second fundamental form $h=(h_{ij})$ is given by the Gaussian formula
\eq{x_{;ij}=-h_{ij}\nu}
and the Weingarten map is denoted by $A=(h^i_j)$.

For any $q\in N$ the pointed Euclidean as well as the hyperbolic space $N\bs\{q\}$ is diffeomorphic to $(0,\8)\x \S^n$ and is covered by geodesic polar coordinates. The metric is given by
\eq{\-g_{\a\b}=dr^2+\vt^2(r)\s_{ij}\equiv dr^2+\-g_{ij},}
where $r$ is the geodesic distance to $q$, $\s$ is the round metric on $\S^n$
and
\eq{\vt(r)=\begin{cases} r, &N=\R^{n+1}\\
						\sinh r, &N=\H^{n+1}.\end{cases}}
The principal curvatures $\-\k_i$ of the coordinates slices $\{r=\mrm{const}\}$ are given by
\eq{\-\k_i=\fr{\vt'(r)}{\vt(r)}.} Since our hypersurfaces will all be convex, they can be written as graphs in geodesic polar coordinates over $\S^n,$
\eq{M=\{(u(y),y)\cn y\in \S^n\},}
where $u$ is a smooth function on $\S^n.$
Define
\eq{v=\sqrt{1+|Du|^2}=\sqrt{1+\-g^{ij}u_{;i}u_{;j}}.}
In terms of a graph, the second fundamental form of $M$ can be expressed as
\eq{h_{ij}v^{-1}=-u_{;ij}+\-h_{ij},}
cf. \cite[Rem.~1.5.1]{Gerhardt:/2006}.

Let us also make some comments on the speed functions under which the family of embeddings evolve. By \cref{1} these are given by smooth, symmetric functions $F$ on an open, symmetric and convex cone $\G\sub\R^{n}$. It is well known, that such a function can be written as a smooth function of the elementary symmetric polynomials $s_{k}$,
\eq{F=\rho(s_{1},\dots,s_{n}),}
compare \cite{Glaeser:01/1963}. The associated functions to $s_{k}$, which are defined on endomorphisms of the tangent space, are traditionally denoted by $H_{k}$ and given by
\eq{H_{k}(A)=\fr{1}{k!}\fr{d}{dt}\det(I+tA)_{|t=0}}
for all $A\in T^{1,1}(M)$, cf. \cite[equ.~(2.1.31)]{Gerhardt:/2006}. Hence also $F$ can be viewed as a function defined on the endomorphism bundle $T^{1,1}(M)$, i.e.
\eq{F=F(A)=F(h^{i}_{j}).}
We will, however, mostly use a different description of $F$, namely as being defined on two variables $(g,h)$ by setting
\eq{\mc{F}(g,h)=F\br{\fr 12 g^{ik}(h_{jk}+h_{kj})},}
where $(g^{ik})$ is the inverse of the positive definite $(0,2)$-tensor $g$ and $h\in T^{0,2}(M)$. We denote by $\mc{F}^{ij}$ and $\mc{F}^{ij,kl}$ the first and second derivatives of $\mc{F}$ with respect to $h$, i.e.
\eq{\mc{F}^{ij}=\fr{\del \mc{F}}{\del h_{ij}},\quad \mc{F}^{ij,kl}=\fr{\del^{2}\mc{F}}{\del h_{ij}\del h_{kl}}.}
Due to the monotonicity assumption on $F$ as a function of the principal curvatures, the tensor $(\mc{F}^{ij})$ is positive definite at all pairs $(g,h)$, such that $(g^{ik}(h_{kj}+h_{jk}))$ has eigenvalues in $\G$. Compare \cite{Andrews:/2007}, \cite[Ch.~2]{Gerhardt:/2006} and \cite{Scheuer:06/2018} for more details.

We also note that we will in the sequel use the same symbol $F$ for both functions $F$ and $\mc{F}$. This will not cause confusion, since in expressions like $F^{ij}h_{ik}h^{k}_{j}$ it only makes sense to think of $\mc{F}$. 

We always assume that $p>1$, set $\Phi(r)=-r^{-p}$ for $r>0$ and
write 
\eq{\Phi'(r) = \frac{d}{dr}\Phi(r),\quad \dot \Phi = \frac{d(\Phi \circ F)}{dt}.}

\section{Evolution equations}\label{Ev}
The proof of the following evolution equations is given in \cite[Lemma~2.3.4]{Gerhardt:/2006}, \cite[Lemma~3.3.2]{Gerhardt:/2006} and \cite[Lemma~5.8]{Gerhardt:01/1996}.

\Theo{lemma}{EvEq_2}{
Under the flow \eqref{4} the geometric quantities 
\eq{\Phi=-F^{-p},\quad \chi=\fr{v}{\vt(u)}} and $u$
evolve as follows
\eq{\label{401}
\dot \Phi- \Phi' F^{ij}\Phi_{;ij} =& \Phi' F^{ij}h_{ik}h^k_j\Phi+K_N\Phi'F^{ij}g_{ij}\Phi,}
\eq{\dot \chi - \Phi' F^{ij}\chi_{;ij} =& - \Phi' F^{ij}h_{ik}h^k_j\chi-2\chi^{-1} \Phi' 
 F^{ij}\chi_{;i}\chi_{;j}+\{\Phi' F + \Phi\}\frac{\bar H}{n}v \chi,}
 and
\eq{\dot u - \Phi' F^{ij}u_{;ij} =& v^{-1}(1+p)F^{-p}- \Phi' F^{ij}\bar h_{ij},}
 where $\-H$ denotes the mean curvature of the coordinate slice $\{r=u\}.$
}

\Theo{lemma}{EvEq}{
Under the flow \eqref{4} the Weingarten map form evolves by
\eq{\label{Evh}\dot{h}^i_j-\Phi'F^{kl}h^i_{j;kl}&=\Phi'F^{kl}h_{rk}h^r_lh^i_j-(\Phi'F-\Phi)h^i_kh^k_j+K_N(\Phi+\Phi'F)\d^i_j\\
			&\hp{=}-K_N\Phi'F^{kl}g_{kl}h^i_j+\Phi^{kl,rs}h_{kl;j}{h_{rs;}}^i.}
For $K_N=-1$ the tensor 
\eq{b^i_j=h^i_j-\d^i_j}
evolves by
\eq{\label{EvBij}\dot{b}^i_j-\Phi'F^{kl}b^i_{j;kl}&=\Phi'F^{kl}b_{rk}b^r_l b^i_j+\Phi'F^{kl}b_{rk}b^r_l \d^i_j+2\Phi b^i_j\\
			&\hp{=}-\br{\Phi'F-\Phi}b^i_k b^k_j +\Phi^{kl,rs}b_{kl;j}{b_{rs;}}^i	}
and hence $B=b^i_i$ evolves by
\eq{\label{EvB}\dot{B}-\Phi'F^{kl}B_{kl}&=\Phi'F^{kl}b_{rk}b^r_lB+n\Phi'F^{kl}b_{rk}b^r_l+2\Phi B\\
						&\hp{=}-\br{\Phi'F-\Phi}b^i_k b^k_i+\Phi^{kl,rs}b_{kl;i}{b_{rs;}}^i.}
}

\pf{
\eqref{Evh} is given in \cite[Lemma~2.4.3]{Gerhardt:/2006}, hence we only prove \eqref{EvBij}. 
\eq{\dot{b}^i_j-\Phi'F^{kl}b^i_{j;kl}&=\Phi'F^{kl}h_{rk}h^r_lh^i_j-(\Phi'F-\Phi)h^i_kh^k_j-(\Phi+\Phi'F)\d^i_j\\
			&\hp{=}+\Phi'F^{kl}g_{kl}h^i_j+\Phi^{kl,rs}h_{kl;j}{h_{rs;}}^i\\
            &=\Phi'F^{kl}\br{h_{rk}h^r_l-2h_{kl}+g_{kl}}h^i_j+2\Phi'Fh^i_j\\
            &\hp{=}-\br{\Phi'F-\Phi}\br{h^i_k-\d^i_k}\br{h^k_j-\d^k_j}-2\br{\Phi'F-\Phi}\br{h^i_j-\d^i_j}\\
            &\hp{=}-\br{\Phi'F-\Phi}\d^i_j-\br{\Phi'F+\Phi}\d^i_j+\Phi^{kl,rs}h_{kl;j}{h_{rs;}}^i.}
Rearranging gives the result.
}

\section{A counterexample to preserved convexity}\label{CEx}

For contracting flows, i.e. flows of the form
\eq{\dot{x}=-\Phi\nu}
with positive $\Phi$, Andrews-McCoy-Zheng \cite{AndrewsMcCoyZheng:07/2013} gave an example of a (weakly) convex hypersurface in $\R^{n+1}$ which develops a negative principal curvature instantly for a certain class of speeds $\Phi$. By continuity with respect to initial values this shows that for these $\Phi$ one can also find strictly convex initial hypersurfaces which develop negative principal curvatures quickly. 

In this section we briefly sketch that we generally face the same phenomenon in our case of expanding flows. Precisely we will see that for powers $p>1$ of the inverse mean curvature flow convexity might be lost, indicating that a stronger pinching condition as for example in \cref{main_result} is needed. 

We will show that the loss of convexity can occur along the flow
\eq{\label{CE}\dot{x}=\fr{1}{H^{p}}\nu,\quad p>1,}
for simplicity in $\R^{3}$. Contrary to the contracting case, where a more sophisticated speed is needed, here we can make use of the strong concavity of the function $\Phi=-H^{-p}$, which gives an additional negative term in 
\eq{\Phi^{kl,rs}=\fr{p}{H^{p+1}}H^{kl,rs}-\fr{p(p+1)}{H^{p+2}}H^{kl}H^{rs}=-\fr{p(p+1)}{H^{p+2}}g^{kl}g^{rs}.}
Let us briefly recall the method in \cite[Thm.~3]{AndrewsMcCoyZheng:07/2013} how to construct such a convex initial hypersurface. First they construct a local graph using the function
\eq{u(\xi)=\fr{c_{1}}{24}\xi_{1}^{4}+\fr{1}{2}\br{a_{2}+b_{2}\xi_{1}+\fr 12 c_{2}\xi_{1}^{2}}\xi_{2}^{2},}
where $a_{2},b_{2}$ are arbitrary positive numbers and 
\eq{c_{1}=\fr{1}{4},\quad c_{2}=\fr{2b_{2}^{2}}{a_{2}}+\fr{1}{4}}
 There holds $u(0)=Du(0)=0$ and hence at $\xi=0$, compare \cite[equ.~(16),(17),(18)]{AndrewsMcCoyZheng:07/2013},
\eq{h_{ij}=u_{,ij},}
\eq{h_{ij;k}=u_{,ijk}}
and
\eq{h_{ij;kl}=u_{,ijkl}-u_{,ij}u_{,km}{u_{,l}}^{m}-u_{,ki}u_{,jm}{u_{,l}}^{m}-u_{,kj}u_{,im}{u_{,l}}^{m},}
where indices appearing after a comma denote usual partial derivatives.
In the proof of \cite[Thm.~2]{AndrewsMcCoyZheng:07/2013} it is shown that the graph of $u$ over a small ball $B_{r}(0)$ is a convex hypersurface, which is strictly convex in $\xi\neq 0$, and that this graph can be closed up to a convex hypersurface, respecting the strict convexity in $\xi\neq 0$. Due to $H(0)=a_{2}>0$, the hypersurface is strictly mean convex and \eqref{CE} is defined for short time.

It remains to show that under the flow \eqref{CE}, the entry $h_{11}(t,0)$ of the second fundamental form, which equals zero at $t=0$, drops below zero instantly. According to \eqref{Evh} we have at $(t,\xi)=(0,0)$:

\eq{\dot{h}_{11}&=\fr{p}{a_{2}^{p+1}}(h_{11;11}+h_{11;22})-\fr{p(p+1)}{a_{2}^{p+2}}(h_{11;1}+h_{22;1})^{2}\\
		&=\fr{p}{a_{2}^{p+1}}(c_{1}+c_{2})-\fr{p(p+1)}{a_{2}^{p+2}}b_{2}^{2}\\
		&=\fr{p}{a_{2}^{p+2}}\br{\fr{a_{2}}{2}+2b_{2}^{2}-(p+1)b_{2}^{2}}\\
		&=\fr{p}{a_{2}^{p+2}}\br{\fr{a_{2}}{2}+(1-p)b_{2}^{2}}\\
		&<0}
for a suitable arrangement of $a_{2}$ and $b_{2}$, due to $p>1$. Hence $h_{11}$ drops below zero instantly and the example is complete.

\section{The pinching estimates}\label{PE}
We define $b_{ij}=h_{ij}+K_Ng_{ij}$ where $K_N=-1$ or  $K_N=0$. We set $B = b^i_i$ and $\|b\|^2=b^i_j b^j_i$ in both cases.

\Theo{lemma}{pinching}{
Let $T^{*}>0$ and let $x$ be a solution of \eqref{4} on $[0, T^{*})$. If $M_0$ satisfies the pinching condition
\begin{equation} \label{3_new}
\|b\|^2-\frac{1}{n}B^2 < c_0B^2,
\end{equation}
 where $0<c_0=c_0(F,n,p)<\frac{1}{n(n-1)}$ is sufficiently small, then \eqref{3_new} remains valid for all $t \in [0, T^{*})$.
}

\begin{proof} 
While in the Euclidean case the proof is similar to the proof of \cite[Prop.~3.4]{Scheuer:07/2016}  we need further arguments in the hyperbolic case.
We begin with some facts which hold in both cases.
Let $\gamma = \frac{1}{n}+c_0$ and define the quantity
\begin{equation}
z= \|b\|^2-\gamma B^2.
\end{equation}
Then, due to \cref{EvEq}, $z$ satisfies the evolution equation
\begin{equation} \label{7_new}
\begin{aligned}
\dot z -\Phi' F^{kl}z_{;kl}
=&  2 \Phi' F^{kl}b_{rk}b^r_l z  - 2 (\Phi'F-\Phi)\left(b^i_kb^k_jb^j_i - \gamma B\|b\|^2\right)\\ 
		& + 2 \Phi^{kl, rs}b_{kl;i}b_{rs;}^{\ \ \ j}\left(b^i_j-\gamma B\delta^i_j\right) - 2 \Phi'F^{kl}\left(b^i_{j;k}b^j_{i;l}-\gamma B_kB_l\right)\\
& -K_N\left\{4\Phi z - 2nc_0 \Phi'F^{kl}b_{rk}b^r_lB\right\}
\end{aligned}
\end{equation}
and due to \cite[Lemma 2.1]{AndrewsMcCoy:03/2012} we have
\begin{equation} \label{And_deriv_new}
\begin{aligned}
g^{kl}\left(b^i_{j;k}b^j_{i;l}-\frac{1}{n}B_kB_l\right) =& \left\|\nabla\left(A-\frac{1}{n}H\id\right)\right\|^2\ge \frac{2(n-1)}{3n}\|\nabla A\|^2 \\
\ge& \frac{2(n-1)}{n(n+2)}\|\nabla H\|^2= \frac{2(n-1)}{n(n+2)}\|\nabla B\|^2.
\end{aligned}
\end{equation} 

(i) Let us now assume that $K_N=0$.
In view of (\ref{3_new}) we have $z(0, \cdot) <0$.
We want to show that $z(t, \cdot)<0$ for all $0 \le t < T^{*}$. For this we assume that there is a minimal $0<t<T^{*}$ and $x\in M_t$
so that $z(t,x)=\sup z(t, \cdot) =0$.  
From \cite[Lemma 2.3]{AndrewsMcCoy:03/2012} we deduce  that
\begin{equation}
\begin{aligned}
b^i_k b^k_j b^j_i -  \gamma B \|b\|^2&\ge c_0 \gamma \left(1-\sqrt{n(n-1)c_0}\right)B^3 \\ &= c_0\left(1-\sqrt{n(n-1)c_0}\right)\|b\|^2B > 0
\end{aligned}
\end{equation}
in $(t,x)$.

Using \eqref{7_new} we conclude that 
\begin{equation} \label{8_new}
\begin{aligned}
0 \le &- 2c_0 (\Phi'F-\Phi)\left(1-\sqrt{n(n-1)c_0}\right) \|b\|^2B\\ 
 &+ 2 B \Phi^{kl, rs}b_{kl;i}b_{rs;}^{\ \ \ j} \left(B^{-1}b^i_j-\gamma\delta^i_j\right) \\
& -\frac{2(n-1)}{n(n+2)}\Phi'\|\nabla B\|^2-\frac{2(n-1)}{3n}\Phi'\|\nabla b\|^2 \\
& - 2 \Phi'\left(F^{kl}-g^{kl}\right)\left(b^i_{j;k}b^j_{i;l}-\frac{1}{n} B_kB_l\right)\\
&+2 c_0 \Phi' \|\nabla B\|^2+2c_0 \Phi' \left(F^{kl}-g^{kl}\right)B_kB_l
\end{aligned}
\end{equation}
in $(t,x)$. Due to \cite[equ.~(4.3)]{AndrewsMcCoy:03/2012},
\begin{equation}
\|F^{kl}-g^{kl}\|\leq \mu\sqrt{c_{0}}
\end{equation}
for a constant $\mu$ only depending on $n$ and $F$, so the terms in the second, fourth and fifth line of (\ref{8_new}) can be absorbed by the terms
in the third line of (\ref{8_new}). Then the right-hand side of (\ref{8_new}) is negative, a contradiction. Note, to estimate the term in the second line of (\ref{8_new}) we used  the homogeneity of $F$,  
\begin{equation}
\Phi^{kl,rs} = \Phi'F^{kl,rs} + \Phi'' F^{kl}F^{rs}
\end{equation}
and
\begin{equation}
c_1 \le \frac{F}{H} \le c_{2}
\end{equation}
in $(t, x)$ where $c_i$ are positive constants depending only on $c_0$. For further details we refer to \cite[equ.~(4.2), (4.3)]{AndrewsMcCoy:03/2012}.

(ii) We assume that $K_N=-1$. The quantity
\begin{equation}
\tilde z = z+\alpha e^{-\Lambda t},
\end{equation}
where a small $\alpha>0$ and a large $\Lambda>0$ will be specified later, satisfies
the evolution equation
\begin{equation} \label{17_new}
\begin{aligned}
\dot {\tilde z} -\Phi' F^{kl}\tilde z_{;kl}
=& -\Lambda \alpha e^{-\Lambda t}+\dot z - \dot \Phi F^{kl}z_{kl}\\
=&  -\Lambda \alpha e^{-\Lambda t}+2 \Phi' F^{kl}b_{rk}b^r_l (\tilde z-\alpha e^{-\Lambda t}) \\
& - 2 (\Phi'F-\Phi)\left(b^i_kb^k_jb^j_i - \gamma B\|b\|^2\right) \\
& + 2 \Phi^{kl, rs}b_{kl;i}b_{rs;}^{\ \ \ j}\left(b^i_j-\gamma B\delta^i_j\right)  - 2 \Phi'F^{kl}\left(b^i_{j;k}b^j_{i;l}-\gamma B_kB_l\right)\\
& -K_N\left\{4\Phi (\tilde z-\alpha e^{-\Lambda t}) - 2nc_0 \Phi'F^{kl}b_{rk}b^r_lB\right\}.
\end{aligned}
\end{equation}

Assuming that $\alpha$ is small we have $\tilde z(0, \cdot) <0$ in view of (\ref{3_new}).
We want to show that $\tilde z(t, \cdot)<0$ for all $0 \le t < T^{*}$. For this we assume that there is a minimal $0<t<T^{*}$ and $x\in M_t$
so that $\tilde z(t,x)=\sup \tilde z(t, \cdot) =0$. 
Due to minimality of $t$ we conclude $B(t',\cdot) > 0$ for all $0 \le t' \le t$. Especially,
\begin{equation}\label{30_new}
\|b\|^2-\frac{1}{n}B^2 < c_0 B^2=c_0(H-n)^2\le c_0 H^2
\end{equation}
and
$M_{t'}$ is strictly horospherically convex for all $0 \le t' \le t$ due to \cite[Lemma~2.2]{AndrewsMcCoy:03/2012}. 
We deduce that 
\begin{equation}
0 < n < F \quad \text{in }[0, t].
\end{equation}
% In view of (\ref{30_new}) we also have
% \begin{equation}
% 0 < c^{-1} < H < c \quad \text{in }[0, t]
% \end{equation}
% so that 
% \begin{equation}
% 0 < B=H-n<H<c \quad \text{in }[0, t].
% \end{equation}

There holds
\begin{equation}
\|b\|^2-\frac{1}{n}B^2 = \tilde c_0B^2 \quad \text{in } (t,x)
\end{equation}
with
\begin{equation}
\tilde c_0 = c_0  -\frac{\alpha e^{-\Lambda t}}{B^2(t,x)}.
\end{equation}

Note, that $0 \le \tilde c_0 < c_0$.
From \cite[Lemma 2.3]{AndrewsMcCoy:03/2012} we deduce  that
\begin{equation}
\begin{aligned}
b^i_k b^k_j b^j_i - & \left(\frac{1}{n}+\tilde c_0\right) B \|b\|^2 >0
\end{aligned}
\end{equation}
in $(t,x)$ if $\tilde c_0>0$. If $\tilde c_0=0$ then $x$ is umbilical point of $M_t$, so let us write $\kappa=\kappa_i$. Then we have 
\begin{equation}
c_0n^2(\kappa-1)^2 = c_0B^2 = \alpha e^{-\Lambda t}
\end{equation}
and
\begin{equation}
\begin{aligned}
b^i_k b^k_j b^j_i - & \left(\frac{1}{n}+ c_0\right) B \|b\|^2 =-c_0n^2(\kappa-1)^3 = -\alpha e^{-\Lambda t}(\kappa-1)
\end{aligned}
\end{equation}
in $(t,x)$.
Using \eqref{17_new}, the maximum principle, \eqref{And_deriv_new} and the fact that 
$B(t,x)>0$ we conclude that 
\begin{equation} \label{18_new}
\begin{aligned}
0 \le & -\Lambda \alpha e^{-\Lambda t}-2\alpha e^{-\Lambda t}\Phi'F^{kl}b_{rk}b^r_l  + I_1 \\
& + 2 B \Phi^{kl, rs}b_{kl;i}b_{rs;}^{\ \ \ j} \left(B^{-1}b^i_j-\gamma\delta^i_j\right) \\
& -\frac{2(n-1)}{n(n+2)}\Phi'\|\nabla B\|^2-\frac{2(n-1)}{3n}\Phi'\|\nabla b\|^2 \\
& - 2 \Phi'\left(F^{kl}-g^{kl}\right)\left(b^i_{j;k}b^j_{i;l}-\frac{1}{n} B_kB_l\right)\\
&+2 c_0 \Phi' \|\nabla B\|^2+2c_0 \Phi' \left(F^{kl}-g^{kl}\right)B_kB_l - 4\alpha \Phi e^{-\Lambda t} 
\end{aligned}
\end{equation}
in $(t,x)$ where
\begin{equation}
I_1 = 
\begin{cases}
2(c_0-\tilde c_0)(\Phi'F-\Phi)B\|b\|^2 \quad &\text{if } \tilde c_0>0, \\
2(\Phi'F-\Phi)\alpha e^{-\Lambda t}(\kappa-1) \quad &\text{if } \tilde c_0=0. 
\end{cases}
\end{equation} 
For $c_0$ sufficiently small, the term 
\begin{equation}
\|F^{kl}-g^{kl}\|
\end{equation}
is small and the terms in the second, fourth and fifth line of (\ref{18_new}) can be absorbed by the terms
in the third line of (\ref{18_new}). Then the right-hand side of (\ref{18_new}) is negative if $\Lambda$ is large, a contradiction. 
\end{proof}

\section{The Euclidean case}\label{EC}
The aim of this section is to prove Theorem \ref{main_result} for the case $K_N=0$. Due to the pinching estimates of the previous section we are in the situation that the proof in \cite{Gerhardt:01/2014} basically carries over literally. For convenience of the reader, and since several elements of the proof simplify due to our pinching estimates, we give an outline of the arguments involved. Throughout this section it is understood that the assumptions of \cref{main_result} hold.

We recall some simple observations. If the initial hypersurface $M_0$ is a sphere of radius $r_0>0$, i.e. $u(0, \cdot)=r_0$, then the flow hypersurfaces of the flow \eqref{4} remain spheres and for their radii $\Theta(t)$ at time $t$ we obtain the ODE
\begin{equation}\label{SphFlow}
\dot \Theta = \frac{\Theta^p}{n^p}, \quad r(0) = r_0,
\end{equation}
with solution 
\begin{equation}
\Theta(t) = \left(\frac{1-p}{n^p}t+r_0^{1-p}\right)^{\frac{1}{1-p}}
\end{equation}
on $[0, T^{*}(r_0))$ where
\begin{equation}\label{sphericaltime}
T^{*}(r_0) = \frac{n^p}{p-1}r_0^{1-p}.
\end{equation}
Hence the spherical flow blows up at time $T^{*}(r_0)$.
The existence of a smooth solution to \eqref{4} up to a maximal time is well known, cf. \cite[Sec.~2.5, Sec.~2.6]{Gerhardt:/2006}.

From the avoidance principle we conclude the following corollary.
\begin{cor} \label{13}
If $r_1, r_2$ are positive constants  so that
\begin{equation}
r_1 < |x(0, \cdot)| < r_2
\end{equation}
where $x$ is the solution of (\ref{4})
then we have
\begin{equation} \label{31}
\Theta(t, r_1) < |x(t, \cdot)| < \Theta(t, r_2) \quad \forall \ 0 \le t < \min\{T^{*}, T^{*}(r_1), T^{*}(r_2)\}.
\end{equation}
\end{cor}

The next aim is to show that $\max|x(t,\cdot)|$ blows up, when the time approaches $T^*.$

\begin{lemma}\label{blowup}
The flow (\ref{4}) only exists in a finite time interval $[0, T^{*})$ and there holds
\begin{equation} \label{33}
 \limsup_{t \rightarrow T^{*}}\max_{M}|x(t, \cdot)| = \infty.
\end{equation}
\end{lemma}
\begin{proof}
 In view of Lemma \ref{pinching} the hypersurfaces are convex and from
 (\ref{31}) we deduce that the maximal time $T^{*}$ has to be finite. Due to the convexity we may write the flow hypersurfaces $M_t$ as radial graphs over the sphere,
 \eq{M_t=\{u(t,x)x\cn x\in \S^n\}}
 for some $u\in C^{\8}([0,T^*)\x\S^n).$ Then $u$ satisfies the scalar flow equation
\begin{equation}
 \dot u = \frac{1}{vF^p}
\end{equation}
where the dot indicates the total time derivative, or
\begin{equation} \label{32}
 \frac{\partial u}{\partial t} = \frac{v}{F^p}
\end{equation}
when we consider the partial time derivative, cf. \cite[p.~98-99]{Gerhardt:/2006}. Under the assumption that $|x|$ is bounded, which is equivalent to $u\leq c,$ we also obtain the $C^1(\S^n)$-estimate
\eq{v\leq c}
due to the convexity of $M_t$ and \cite[Thm.~2.7.10]{Gerhardt:/2006}. To obtain a $C^2(\S^n)$-estimate, we need some curvature estimates. The proof is similar to the one for \cite[Lemma~3.10, Lemma.~4.4]{Gerhardt:01/2014}: Define the auxiliary function
\eq{w=\log(-\Phi)+\log\chi +\g u,}
which, due to \cref{EvEq_2}, satisfies the evolution equation
\eq{\label{blowup-1}\dot{w}-\Phi'F^{ij}w_{;ij}&=\Phi'F^{ij}(\log(-\Phi))_{;i}(\log(-\Phi))_{;j}-\Phi'F^{ij}(\log\chi)_{;i}(\log\chi)_{;j}\\
					&\hp{=}+(\Phi'F+\Phi)\fr{\bar{H}}{n}v+\g(\Phi'F-\Phi)v^{-1}-\g\Phi'F^{ij}\bar{h}_{ij}.}
At spatial maxima of $w$ there holds (recall $\chi=\fr vu,$)
\eq{\label{blowup-2}&\Phi'F^{ij}(\log(-\Phi))_{;i}(\log(-\Phi))_{;j}-\Phi'F^{ij}(\log\chi)_{;i}(\log\chi)_{;j}\\
	=~&2\g\Phi'F^{ij}(\log\chi)_{;i}u_{;j}+\g^2\Phi'F^{ij}u_{;i}u_{;j}\\
    =~&-2\g v\Phi'F^{ij}h^k_i u_{;k}u_{;j}+\g^2\Phi'F^{ij}u_{;i}u_{;j},}
    where we used
\eq{v_{;i}=-v^2 h^k_i u_{;k}+v\fr{\-H}{n}u_{;i},}    
cf. \cite[equ.~(5.29)]{Gerhardt:01/1996}. Since 
    \eq{F^{ij}\-h_{ij}=\fr{\-H}{n}F^{ij}\-g_{ij}\geq c F^{ij}g_{ij}\geq cn,}
    due to the pinching estimate,
we have
\eq{\g^2\Phi'F^{ij}u_{;i}u_{;j}-\fr{\g}{2}\Phi'F^{ij}\-h_{ij}\leq \g\Phi'F^{ij}\br{\g u_{;i}u_{;j}-\fr c2 g_{ij}}<0}
for small $\g$
and the other half of the term $F^{ij}\-h_{ij}$ can be used to absorb the other positive terms in \eqref{blowup-1}, since the negatively signed term has the highest order in $\fr{1}{F}$. We obtain that $w$ is a priori bounded. It is immediate from \cref{EvEq_2} and the maximum principle that $F$ is bounded as well. Due to the pinching estimates \cref{pinching} we have 
    \eq{\k_1\geq c\k_n,}
also compare \cite[Lemma~2.2]{AndrewsMcCoy:03/2012}. Hence $\k_n$ must be bounded and $\k_1\geq c>0$ as long as $u$ is bounded. Thus as long as the flow ranges in compact subsets of $\R^{n+1}$ we have uniform $C^2(\S^n)$-estimates. Note that we can not use the Krylov-Safonov theory to deduce $C^{2,\alpha}(\S^{n})$ estimates, since we did not assume any sign on the second derivatives of $F$.

However, due to the pinching estimates we are in the situation that $F^{ij}$ is as close to $g^{ij}$ as we want and hence we can use a parabolic version of the $C^{1,\alpha}$-estimates originally proved by Cordes \cite{Cordes:06/1956} and Nirenberg \cite{Nirenberg:/1954} for linear elliptic equations, cf. 
\cite[Lemma 12.13]{Lieberman:/1998}. This theorem is also stated in \cite[Thm. 7.3]{AndrewsMcCoy:03/2012}, where in addition the reader may find the detailed procedure, how one can get H\"older estimates on the second derivatives of solutions to the curvature flow equation.
Higher order estimates then follow from Schauder theory. Hence at the maximal time of existence $u$ must blow up.    
\end{proof}

Let $r_0>0$ be so that $T^{*}(r_0)=T^{*}$, where $T^{*}(r_{0})$ is given by \eqref{sphericaltime}, then for all $0 \le t < T^{*}$ there is a $\xi_t\in \mathbb{S}^n$ such that
\begin{equation} \label{500}
u(t, \xi_t) = \Theta(t, r_0)
\end{equation}
in view of Corollary \ref{13}. The spherical flow with existence time $T^*(r_0)$ provides the correct scaling factor. In particular, applying the evolution equation for the support function
$\-u=\chi^{-1},$
as deduced in \cite[Section 2]{Urbas:/1991} and the oscillation estimates from \cite[Theorem 3.1]{ChowGulliver:/1996}, also compare \cite[Theorem 3.1]{McCoy:/2003}, we conclude that 
the following lemma holds for
\eq{\~u=u\Theta^{-1}.}

\begin{lemma} \label{602}
 Let $u$ be the solution of the scalar flow equation (\ref{32}). 
 Then there exists a positive constant $c$ such that
 \begin{equation} \label{230}
  u(t, x) -c \le \Theta(t, r_0) \le u(t, x) +c \quad \forall x \in \mathbb{S}^n,
 \end{equation}
 hence
 \begin{equation} \label{201}
  \lim_{t \rightarrow T^{*}} \~u(t,x) =1 \quad \forall x \in \mathbb{S}^n.
 \end{equation}
 We also have
 \begin{equation}
v-1 \le c \Theta^{-1}.
\end{equation}

\end{lemma}
\begin{proof}
Literally as in \cite[Lemma~3.5, Lemma~3.7]{Gerhardt:01/2014}.
\end{proof}

We need the following bounds for the rescaled principal curvatures.

\Theo{lemma}{Scaledkappa}{
Under the flow \eqref{4}, the rescaled principal curvatures 
$\~\k_i=\Theta \k_i$
satisfy
\eq{0<c_1\leq \~\k_i\leq c_2}
for suitable constants $c_i=c_i(M_0).$
}

\pf{
Due to the pinching estimates, as in the proof of \cref{blowup} it suffices to prove the existence of constants $c_i$ such that 
\eq{0<c_1\leq \Theta F=\~F=F(\~h^i_j)\leq c_2.}
We begin with the upper bound. For $\l>0$ define the function
\eq{z=\log\~F^p+\l\~u\equiv-\log(-\~\Phi)+\l\~u.}
Due to \cref{EvEq_2} and \eqref{SphFlow}, $z$ satisfies
\eq{\dot{z}-\Phi'F^{ij}z_{;ij}&\leq-\Phi'F^{ij}h_{ik}h^k_j+\fr{p}{n^p}\Theta^{p-1}+\l v^{-1}(p+1)\~F^{-p}\Theta^{p-1}\\
					&\hp{=}-\l p\~F^{-(p+1)}F^{ij}\-g_{ij}\~u^{-1}\Theta^{p-1}-\l n^{-p}\~u\Theta^{p-1}\\
                    &\leq \fr{\Theta^{p-1}}{n^p}\br{p-\fr{\l}{2}\~u}+\l\Theta^{p-1}\br{v^{-1}(p+1)\~F^{-p}-\fr{n^{-p}}{2}\~u},}
which is negative if $\l$ and $\~F$ are large enough, due to \cref{602}.

To prove the lower bound we proceed as in \cite[Lemma~3.10]{Gerhardt:01/2014}.
We consider the function
  \begin{equation}
  w = \log(-\tilde \Phi) + \log \tilde \chi + \tilde u,
 \end{equation}
 where $\~\chi=\Theta\chi.$
 Let $0<T<T^{*}$ be arbitrary and assume that $\sup_{Q_T}w$, where
 \begin{equation}
 Q_T= [0, T] \times \mathbb{S}^n,
 \end{equation} 
 is attained in $(t_0, x_0)$ with large $t_0>0$. The maximum principle implies in $(t_0, x_0)$ 
 \begin{equation} \label{301}
 \begin{aligned}
  0 \le&  \Phi' F^{ij}(\log(-\tilde \Phi ))_i(\log(-\tilde \Phi ))_j  - \Phi' F^{ij}(\log \tilde \chi)_i(\log \tilde \chi)_j \\
  & + c \tilde F^{-p}\Theta^{p-1}-p\tilde F^{-(p+1)}F^{ij}\bar g_{ij}\tilde u^{-1}\Theta^{p-1}
  \end{aligned}
 \end{equation}
 where we assume w.l.o.g. that $\tilde F$ is small. The fourth term on the right-hand side of (\ref{301}) is dominating the third term. From $w_i=0$ we conclude
 \begin{equation}\label{730}
 \begin{aligned}
  &\Phi' F^{ij}(\log(-\tilde \Phi))_i(\log(-\tilde \Phi))_j- \Phi' F^{ij}(\log \tilde \chi)_i(\log \tilde \chi)_j 
\\
  =~&\Phi' F^{ij}\tilde u_i\tilde u_j + 2 \Phi' F^{ij}(\log \tilde \chi)_i \tilde u_j.
 \end{aligned}
 \end{equation}
 Here the first term on the right-hand side is of order $\tilde F^{-(p+1)}$ but $\|Du\|$ vanishes if $t$ tends to $T^{*}$ while the second term is nonpositive for the same reason as in \eqref{blowup-2}.
 Hence $w$ is a priori bounded from above.
}

The proofs of Theorem \ref{main_result} (i) and Theorem \ref{main_result} (ii) in the Euclidean case can be completed literally as in \cite[Lemma~5.1]{Gerhardt:01/2014}, up to replacing Krylov-Safonov by the Cordes-Nirenberg type result as above.

The proof of Theorem \ref{main_result} (iii) 
follows exactly the arguments in \cite{Scheuer:07/2016} by
using Theorem \ref{main_result} (i) and Theorem \ref{main_result} (ii) instead of \cite{Gerhardt:01/2014}
as it is used in \cite{Scheuer:07/2016}.

\section{The hyperbolic case}\label{HC}

Due to the pinching estimates in the hyperbolic case, \cite[Lemma~2.2]{AndrewsMcCoy:03/2012} implies that the flow hypersurfaces remain strictly horospherically convex. Hence several aspects of the proofs in \cite{Scheuer:05/2015} simplify. We can easily prove the following lemma with the help of a shortcut compared to \cite[Thm.~4.4]{Scheuer:05/2015}.

{\thm{\label{HypDecay}
The flow \eqref{flow} exists for all times, remains strictly horospherically convex and the principal curvatures of the flow hypersurfaces of \eqref{flow} converge to $1$ exponentially fast,
\eq{|\k_{i}-1|\leq ce^{-\fr{2}{n^{p}}t}\quad\forall 0\leq t<\infty,}
where $c=c(n,p,M_{0})$.
}}

\pf{From the evolution of the flow speed
\eq{-\Phi=\fr{1}{F^{p}},}
\eq{\label{HypDecay-1}\dot \Phi- \Phi' F^{ij}\Phi_{;ij} = \Phi' F^{ij}\br{h_{ik}h^k_j-g_{ij}}\Phi,}
cf. \cref{EvEq_2}, we obtain that $\max F$ is strictly decreasing. Hence by the pinching estimates all principal curvatures are bounded and due to the convexity of the flow hypersurfaces we also have uniform gradient estimates using \cite[Thm.~2.7.10]{Gerhardt:/2006}. Since in finite time the flow remains in a compact subset of $\H^{n+1}$, we obtain the long time existence similarly as in the Euclidean case. We can estimate \eqref{HypDecay-1} with the help of the pinching estimate:
\eq{\fr{d}{dt} (-\Phi)- \Phi' F^{ij}(-\Phi)_{;ij} &= \Phi' F^{ij}\br{h_{ik}h^k_j-g_{ij}}(-\Phi)\\
			&\geq c\br{\fr 1n H^{2}-n}\\
			&\geq c(H-n).}
Hence $\max F$ must converge to $n$, for otherwise there existed $\delta>0$ with 
\eq{F(t,\xi_{t})=\max_{M_{t}}F\geq n+\delta}
for all $t>0$ and hence
\eq{\k_{n}(t,\xi_{t})\geq 1+\tilde{\delta}}
for some $\tilde{\delta}>0$. But then
\eq{c\tilde{\delta}\leq \fr{d}{dt}\br{\min_{M_{t}}F^{-p}}}
for almost every $t>0$ and $\min F^{-p}$ would converge to infinity, which is impossible. But this implies $\max \k_{1}\ra 1$ and due to the pinching estimates \cref{pinching} we have 
\eq{\max \k_{n}\ra 1.}

Set 
\eq{w=\left(\Phi+n^{-p}\right)e^{\fr{\alpha}{n^{p}} t},}
where $0<\alpha<2$.
From \eqref{401} we deduce 
\begin{equation}\label{evolution_w}
\begin{aligned}
\dot{w}-\Phi'F^{kl}w_{kl}&=\Phi'F^{ij}\left(h_{ik}h^k_j-g_{ij}\right)\Phi e^{\frac{\alpha}{n^p}t}+\frac{\alpha}{n^p}e^{\frac{\alpha}{n^p}t}\left(\Phi+n^{-p}\right).
\end{aligned}
\end{equation}
Writing $\kappa_i=1+\delta_i$ with appropriate small $\delta_i\ge 0$ we obtain the following first order expansions
\eq{|A|^2-n=2\sum_{i=1}^{n}\delta_i+O(\delta_n^{2})}
and 
\eq{\Phi+n^{-p}=\frac{p}{n^{p+1}}\sum_{i=1}^n\delta_i+O(\delta_n^{2}).}
Furthermore, we have $\Phi' \Phi=-pF^{-2p-1}$ so that the leading term in (\ref{evolution_w}) for large $t$
is
\begin{equation} \label{evolut_error_term}
e^{\frac{\alpha}{n^p}t}\frac{p}{n^{2p+1}}\sum_{i=1}^n\delta_i(-2+\alpha).
\end{equation}
Hence $w$ is bounded for any $\alpha<2$ and thus
\eq{\delta_{n}\leq ce^{-\fr{\alpha t}{n^{p}}}.}

In the second step set
\eq{\~w=\left(\Phi+n^{-p}\right)e^{\fr{2}{n^{p}}t}}
and deduce 
\eq{\dot{\~w}-\Phi'F^{kl}\~w_{;kl}\leq c\delta_n^2e^{\frac{2}{n^p}t}\le c e^{\frac{2-2\alpha}{n^p}t},} 
where we used that the corresponding first order terms in the evolution equation for $\tilde w$, compare (\ref{evolut_error_term}), now vanish.
 The result follows from the maximum principle since $2-2\alpha<0$, e.g. for $\alpha=3/2$.
}

In order to show that also the gradient of the graph functions $u$ converges to zero exponentially fast, we use the conformally flat parametrization of the hyperbolic space.

{\lemma{\label{GradDecay}
Let the $M_{t}$ be expressed as graphs in geodesic polar coordinates,
\eq{M_{t}=\graph_{\S^{n}} u(t,\cdot),}
then the quantity 
\eq{v=\sqrt{1+\-g^{ij}u_{i}u_{j}}}
satisfies
\eq{|Du|^{2}=v^{2}-1\leq ce^{-\fr{2}{n^{p}}t}.}
}}

\pf{
Defining the radial coordinate $\rho$ by
\eq{r=\log(2+\rho)-\log(2-\rho),}
the hyperbolic space can be parametrised over the ball $B_{2}(0)$ to yield
\eq{\-g=\fr{1}{\br{1-\fr 14\rho^{2}}^{2}}\br{d\rho^{2}+\rho^{2}\s_{ij}dx^{i}dx^{j}}=e^{2\psi}\br{d\rho^{2}+\rho^{2}\s_{ij}dx^{i}dx^{j}}.}
Viewing the flow hypersurfaces $M_{t}\sub\H^{n+1}$ as hypersurfaces $\~M_{t}\sub B_{2}(0)$ of the Euclidean space, the second fundamental forms are related by
\eq{\label{GradDecay-1}e^{\psi} h^{i}_{j}=\~h^{i}_{j}+\psi_{\b}\~\nu^{\b}\d^{i}_{j},}
cf. \cite[Prop.~1.1.11]{Gerhardt:/2006}.
Hence
\eq{\|\~A\|^{2}-\fr 1n \~H^{2}=e^{2\psi}\br{\|A\|^{2}-\fr 1n H^{2}}\leq ce^{-\fr{2}{n^{p}}t},}
due to \cref{HypDecay} and since
\eq{e^{\psi}=\fr{4}{(2+\~u)(2-\~u)}\leq ce^{\fr{t}{n^{p}}},}
where we also used the existence of a constant $c$ such that
\eq{-c<u-\fr{t}{n^{p}}<c,}
cf. \cite[Cor.~3.3, Lemma~3.5]{Scheuer:05/2015}.
Hence the conformal hypersurfaces $\~M_{t}$ become eventually strictly convex (also note $\~H>0$ using \eqref{GradDecay-1}) and they converge uniformly to $\del B_{2}(0).$ Using 
\eq{\~v\leq ce^{\-\k\cdot\mrm{osc}(\~u)}}
again, \cite[Thm.~2.7.10]{Gerhardt:/2006}, we obtain $\~v\ra 1.$
But 
\eq{|D\~u|^{2}=\~u^{-2}\s^{ij}\~u_{i}\~u_{j}=|Du|^{2},}
as a simple calculation reveals, and hence $v\ra 1$ for the $M_{t}.$ From the proof of \cite[Thm.~4.1]{Scheuer:05/2015} and especially equ. (4.6) in this proof we obtain
\eq{v-1\leq ce^{-\l t}}
for suitable $\l>0$. Furthermore, using the proof of \cite[Thm.~4.5]{Scheuer:05/2015} literally, we obtain the claim.
}

The proof of \cref{main_result} can now be finished as in \cite{Scheuer:05/2015}. Let us shortly sketch the strategy. Define a rescaling of $u$, 
\eq{\p=\int_{r_{0}}^{u}\fr{1}{\vt(s)}~ds,}
where $r_{0}<\inf_{M}u(0,\cdot)$ and $\vt=\sinh$. There holds
\eq{|Du|^{2}=\s^{ij}\p_{i}\p_{j}=|D\p|^{2}\leq ce^{-\fr{2}{n^{p}}t}.}
Due to 
\eq{h^{i}_{j}=v^{-1}\vt^{-1}\br{\vt'\d^{i}_{j}-\br{\s^{ik}-v^{-2}\p^{i}\p^{k}}\p_{kj}},}
where index raising and covariant differentiation are performed with respect to $\s_{ij},$ cf. \cite[(3.26)]{Gerhardt:11/2011}, we also obtain
\eq{|D^{2}\p|\leq ce^{-\fr{t}{n^{p}}},}
also compare \cite[Thm.~4.6]{Scheuer:05/2015}. Now Sections 5 and 6 of \cite{Scheuer:05/2015} finish the proof verbatim, where Section 5 gives higher order estimates of the form
\eq{|D^{m}\p|\leq ce^{-\fr{t}{n^{p}}}}
by differentiating the equation satisfied by $D\p$ and Section 6 relates this back to the original function $u$ and its rescaled versions.

\bibliographystyle{amsplain}
\bibliography{Bibliography.bib}

\end{document}